\documentclass[10pt,a4paper]{amsart}
\usepackage{graphicx,multirow,array,amsmath,amssymb}

\newtheorem{theorem}{Theorem}
\newtheorem{lemma}[theorem]{Lemma}

\begin{document}

\title{Upper bounds on the minimal length of cubic lattice knots}
\author[K. Hong]{Kyungpyo Hong}
\address{Department of Mathematics, Korea University, Anam-dong, Sungbuk-ku, Seoul 136-701, Korea}
\email{cguyhbjm@korea.ac.kr}
\author[S. No]{Sungjong No}
\address{Department of Mathematics, Korea University, Anam-dong, Sungbuk-ku, Seoul 136-701, Korea}
\email{blueface@korea.ac.kr}
\author[S. Oh]{Seungsang Oh}
\address{Department of Mathematics, Korea University, Anam-dong, Sungbuk-ku, Seoul 136-701, Korea}
\email{seungsang@korea.ac.kr}

\thanks{This work was supported by the National Research Foundation of Korea(NRF)
grant funded by the Korea government (MEST) (No.~2009-0074101).}

\begin{abstract}
Knots have been considered to be useful models for simulating molecular chains such as DNA and proteins.
One quantity that we are interested on molecular knots is
the minimum number of monomers necessary to realize a knot.
In this paper we consider every knot in the cubic lattice.
Especially the minimal length of a knot indicates the minimum length
necessary to construct the knot in the cubic lattice.
Diao introduced this term (he used ``minimal edge number" instead) and
proved that the minimal length of the trefoil knot $3_1$ is $24$.
Also the minimal lengths of the knots $4_1$ and $5_1$ are known to be $30$ and $34$, respectively.
In the article we find a general upper bound of the minimal length of a nontrivial knot $K$,
except the trefoil knot, in terms of the minimal crossing number $c(K)$.
The upper bound is $\frac{3}{2}c(K)^2 + 2c(K) + \frac{1}{2}$.
Moreover if $K$ is a non-alternating prime knot, then the upper bound is
$\frac{3}{2}c(K)^2 - 4c(K) + \frac{5}{2}$.
Our work are considerably direct consequences of the results done by the authors in \cite{HNO}.
Furthermore if $K$ is $(n+1,n)$-torus knot,
then the upper bound is $6 c(K) + 2 \sqrt{c(K)+1} +6$.
\end{abstract}

\maketitle

\section{Introduction} \label{sec:intro}

A knot is a closed curve in $3$-space $\mathbb{R}^3$.
Knots are commonly found in molecular chains such as DNA and proteins,
and they have been considered to be useful models for structural analysis of these molecules.
A knot can be embedded in many different ways in $3$-space, smooth or piecewise linear.
Polygonal knots are those which consist of finite line segments, called {\em sticks\/},
attached end-to-end.
This representation of knots is very useful for many applications in Science.
The microscopic level molecules are more similar to rigid sticks than a flexible rope.
In fact the DNA strand is made up of small rigid sticks of sugar,
phosphorus, nucleotide proteins, and hydrogen bonds.
Chemists are also interested in knotted molecules which are formed by a sequence of atoms
bonded end-to-end so that the last one is also bonded to the first.

A question that naturally arises from these studies is the following.
What is the smallest number of atoms needed to construct a nontrivial knotted molecule?
This question has been addressed by a theoretical study of polygonal knots in lattices \cite{D1}.
Here we aim to determine the minimum length needed to construct a knot in the simple cubic lattice.
A {\em lattice knot\/} is a polygonal knot in the cubic lattice
$\mathbb{Z}^3=(\mathbb{R} \times \mathbb{Z} \times \mathbb{Z}) \cup (\mathbb{Z}
\times \mathbb{R} \times \mathbb{Z}) \cup (\mathbb{Z} \times \mathbb{Z} \times \mathbb{R})$.
For further studies on lattice knots the readers are referred to \cite{G, H1, H2, HO1, HO2, HW, K, P, SW}.
An {\em edge\/} is a line segment of unit length joining two nearby lattice points in $\mathbb{Z}^3$.
Obviously a stick with length $n$ of a lattice knot consists of $n$ edges.

The minimum number of edges necessary to realize a knot as a lattice knot
is called the {\em minimal length\/}.
The minimal length was numerically estimated for various knots \cite{HNRAV, JaP, SIADSV}.
Furthermore such estimation was mathematically confirmed for few small knots.
Diao proved rigorously that the minimal length of any nontrivial lattice knot is at least $24$ and
only the trefoil knot $3_1$ can be realized with $24$ edges \cite{D1}.
Later it was proved that the minimal length of $4_1$ and $5_1$ are $30$ and $34$,
respectively \cite{SIADSV}.
Lattice knots with the minimal length of the knots $3_1$, $4_1$,
and $5_1$ are depicted in Figure \ref{fig:1}.

\begin{figure}[h]
\begin{center}
\includegraphics[scale=1.3]{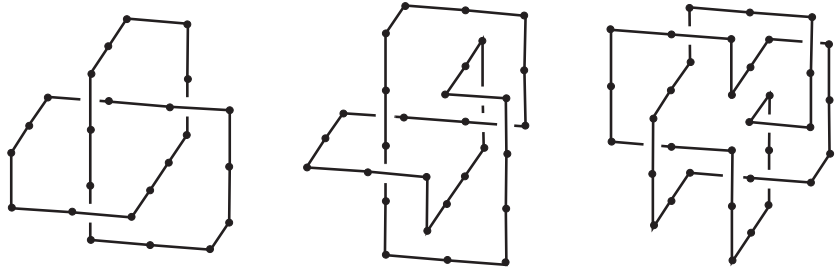}
\end{center}
\caption{$3_1$, $4_1$, and $5_1$ in the cubic lattice}
\label{fig:1}
\end{figure}

In the article we find a general upper bound of the minimal length of a nontrivial
lattice knot $K$ in terms of the minimal crossing number $c(K)$.
The following theorem is the main result of this article.

\begin{theorem} \label{thm:main}
Let $K$ be a nontrivial knot which is not the trefoil knot $3_1$.
Then the minimal length of $K$ is at most $\frac{3}{2}c(K)^2 + 2c(K) + \frac{1}{2}$.
Moreover if $K$ is a non-alternating prime knot,
then the minimal length is at most $\frac{3}{2}c(K)^2 - 4c(K) + \frac{5}{2}$.
Furthermore if $K$ is $(n+1,n)$-torus knot,
then the minimal length is at most $6 c(K) + 2 \sqrt{c(K)+1} +6$.
\end{theorem}

\section{Definitions}

In this section, we introduce several definitions and terminology.
For convenience the notations concerning the $y$ and $z$-coordinates will be defined
in the same manner as the $x$-coordinate.
A stick in $\mathbb{Z}^3$ parallel to the $x$-axis is called an $x$-{\em stick\/}
and an edge parallel to the $x$-axis is called an $x$-{\em edge\/}.
We denote an $x$-stick
$(x_{ii'},j,k) = \{(x,y,z) \in \mathbb{Z}^3 \, | \, i\leq x \leq i', \, y=j, \, z=k \}$
for some integers $i,i',j$ and $k$.
The plane with the equation $x=i$ for some integer $i$ is called an $x$-{\em level} $i$.
So each $y$-stick or $z$-stick lies on an $x$-level $i$ for some $i$.
Note that the $x$-stick $(x_{ii'},j,k)$ whose endpoints lie on $x$-levels $i$
and $i'$ has length $|i'-i|$, and so consists of $|i'-i|$ $x$-edges.
A lattice knot is said to be {\em properly leveled\/} if each $x$-level
({\em resp.} $y$, $z$-level) contains exactly two endpoints of $x$-sticks ({\em resp.} $y$, $z$-sticks).
If a properly leveled lattice knot has $n$ $x$-levels,
then we may say that these are $x$-levels $1,2,\cdots,n$ like heights without changing the knot type.

\begin{lemma} \label{lem:leveled}
If a properly leveled lattice knot has $n$ $x$-levels, then this lattice knot contains at most
$\frac{n^2-1}{2}$ $x$-edges if $n$ is odd, or $\frac{n^2}{2}$ $x$-edges if $n$ is even.
\end{lemma}

\begin{proof}
By proper leveledness there are exactly two $x$-edges between $x$-level $1$ and $2$
(similarly for $n-1$ and $n$).
And there are at most four $x$-edges between $x$-level $2$ and $3$
(similarly for $n-2$ and $n-1$).
Keep going to add two $x$-edges when $x$-levels approach to the middle level.
Then the maximum number of $x$-edges is
$2 \times \sum^{\frac{n-1}{2}}_{i=1} 2i = \frac{n^2-1}{2}$ if $n$ is odd, or
$2 \sum^{\frac{n}{2}-1}_{i=1} 2i + 2 (\frac{n}{2}) = \frac{n^2}{2}$ if $n$ is even.
\end{proof}

There is an open-book decomposition of $\mathbb{R}^3$ which has open half-planes as pages
and the standard $z$-axis as the binding axis.
We may regard each page as a half-plane $H_{\theta}$ at angle $\theta$
when the $x$-$y$ plane has a polar coordinate.
It can be easily shown that every knot $K$ can be embedded in an open-book
with finitely many pages so that it meets each page in a simple arc.
Such an embedding is called an {\em arc presentation\/} of $K$.
And the {\em arc index\/} $a(K)$ is defined to be the minimal number of pages
among all possible arc presentations of $K$.
For example, the left figure in Figure \ref{fig2} shows an arc presentation of figure-$8$ knot
which has the arc index $6$.
Here the points of $K$ on the binding axis are called {\em binding indices\/},
assigned by $1,2, \cdots , a(K)$ from bottom to top.
Also we assign the page numbers $1,2, \cdots , a(K)$
to all of the arcs from the back page to the front.
For further studies on arc presentation the readers are referred to \cite{C}.

\begin{figure} [h]
\begin{center}
\includegraphics[scale=1.2]{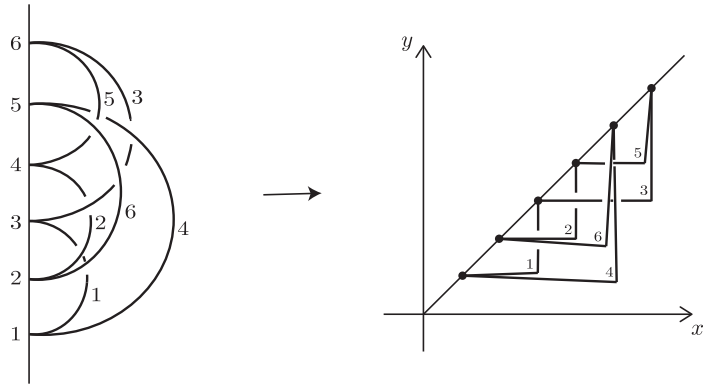}
\end{center}
\caption{Arc presentation and lattice arc presentation} \label{fig2}
\end{figure}

In this paper we find an upper bound of the minimal length in terms of arc index.
The following theorem has key role to convert it to an upper bound in terms of crossing number.
Bae and Park established an upper bound on arc index in terms of crossing number.
In their paper \cite{BP}, Corollary $4$ and Theorem $9$ provide that $a(K) \leq c(K)+2$,
and moreover $a(K) \leq c(K)+1$ if $K$ is a non-alternating prime knot.
Later Jin and Park improved the second part of Bae and Park's theorem.
Theorem $3.3$ in \cite{JiP} provides that if $K$ is a non-alternating prime knot, then $a(K) \leq c(K)$.

\begin{theorem} \label{thm:ac}
Let $K$ be any nontrivial knot. Then $a(K) \leq c(K)+2$.
Moreover if $K$ is a non-alternating prime knot, then $a(K) \leq c(K)$.
\end{theorem}

We move the binding axis to the line $y=x$ on the $x$-$y$ plane and
replace each arc by two connected sticks
which are an $x$-stick and a $y$-stick properly below the line $y=x$.
For better view we slightly perturb each pair of sticks which are overlapped.
The resulting is called a {\em lattice arc presentation\/} of the knot on the plane.
See the right figure in Figure \ref{fig2}.
This lattice arc presentations of knots are very useful to construct lattice knots in $\mathbb{Z}^3$.

\section{Proof of Theorem \ref{thm:main}}

The key idea of the construction of a proper lattice knot is given by the authors in \cite{HNO}.
The following two lemmas are almost direct consequences of Lemma $4$, $5$, and $6$ in the paper.
For completeness we recall the proofs of the lemmas.

\begin{lemma} \label{lem:-2}
Let $K$ be a nontrivial knot with the arc index $a(K)$.
Then we can realize $K$ as a properly leveled lattice knot
by using $a(K)-1$ $x$-sticks, $a(K)-1$ $y$-sticks, and $a(K)$ $z$-sticks.
\end{lemma}

\begin{proof}
We begin with a lattice arc presentation of $K$ with $a(K)$ arcs
as the right figure in Figure \ref{fig2}.
To realize each arc which has the page number $k$ and
the binding indices $i$ and $j$ (assume that $i<j$)
at its endpoints for some $i,j,k = 1,2, \cdots, a(K)$,
we build two sticks $(x_{ij},i,k)$ and $(j,y_{ij},k)$ on $\mathbb{Z}^3$.
And then we connect each pair of arcs which share a binding index, say $i$,
by a $z$-stick $(i,i,z_{kk'})$ where $k$ and $k'$ are the page numbers of the pair.
The resulting lattice knot is properly leveled and
consists of $a(K)$ $x$-sticks, $a(K)$ $y$-sticks, and $a(K)$ $z$-sticks.
See the left figure in Figure \ref{fig3}.

Still we have a chance to reduce $2$ more sticks as follows.
Consider the two $x$-sticks and the $z$-stick lying on $y$-level $1$.
These two $x$-sticks are $(x_{1i},1,k)$ and $(x_{1i'},1,k')$ (assume that $i < i'$ and $k < k'$),
so this $z$-stick is $(1,1,z_{kk'})$.
Delete the shorter $x$-stick $(x_{1i},1,k)$, and replace the other two sticks
by an $x$-stick $(x_{ii'},1,k')$ and a $z$-stick $(i,1,z_{kk'})$.
We repeat this replacement for the two $y$-sticks and the $z$-stick lying on $x$-level $a(K)$.
The final lattice knot is also properly leveled and
consists of $a(K)-1$ $x$-sticks, $a(K)-1$ $y$-sticks, and $a(K)$ $z$-sticks
as the right figure in Figure \ref{fig3}.
\end{proof}

\begin{figure} [h]
\begin{center}
\includegraphics[scale=1.2]{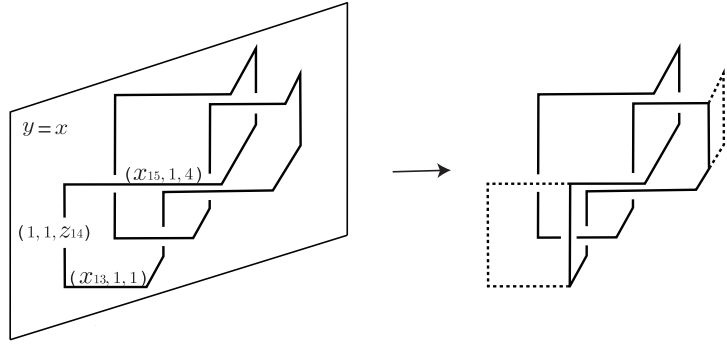}
\end{center}
\caption{Lattice presentation with $2$ sticks reduced} \label{fig3}
\end{figure}

\begin{lemma} \label{lem:-4}
Let $K$ be a nontrivial knot with the arc index $a(K)$.
Suppose that $K$ is not $(n+1,n)$-torus knot.
Then we can realize $K$ as a properly leveled lattice knot
by using $a(K)-2$ $x$-sticks, $a(K)-1$ $y$-sticks, and $a(K)-1$ $z$-sticks.
\end{lemma}

\begin{proof}
The proofs of Lemmas $5$ and $6$ in \cite{HNO} guarantee that
if $K$ is not $(n+1,n)$-torus knot, then we have a special lattice arc presentation of $K$
so that an arc $l$ with the page number $1$ has the binding indices $a$ and $b$,
and another arc $l'$ has the binding indices $b$ and $a(K)$ satisfying $1<a<b<a(K)$.
In more details, this special lattice arc presentation can be obtained from, so called,
a non-star shaped arc presentation of $K$ as the first paragraph in the proof of
Lemma $5$ in \cite{HNO}.
Note that if $K$ has a star shaped arc presentation and is not $(n+1,n)$-torus knot,
then its dual arc presentation, so called, converts into a non-star shaped arc presentation
as the last paragraph in the proof of Lemma $6$ in \cite{HNO}.

Let $k$ be the page number of the arc $l'$.
See the figures in Figure \ref{fig4} for better understanding.
We slightly change the related lattice arc presentation so that only
the arc $l$ lies above the line $y=x$.
This does not change the knot $K$ because $l$ does not touch the other part of $K$
when rotating it $180^{\circ}$ around the binding axis.
Now repeat the whole part of the proof of Lemma \ref{lem:-2}.
Notice that the arc $l$ is not involved in the reduction of two sticks
on $y$-level $1$ and $x$-level $a(K)$
since $l$ does not have the binding indices $1$ and $a(K)$.

Still we can lift $l$ up through the $z$-axis till it reaches the $z$-level $k$
with replacing $z$-sticks properly.
Indeed the new $x$-stick of $l$ lies on the same line containing the $x$-stick of the arc $l'$.
This means that we can delete an $x$-stick and an $z$-stick.
The resulting lattice knot is obviously properly leveled and
consists of $a(K)-2$ $x$-sticks, $a(K)-1$ $y$-sticks, and $a(K)-1$ $z$-sticks.
\end{proof}

\begin{figure} [h]
\begin{center}
\includegraphics[scale=1.1]{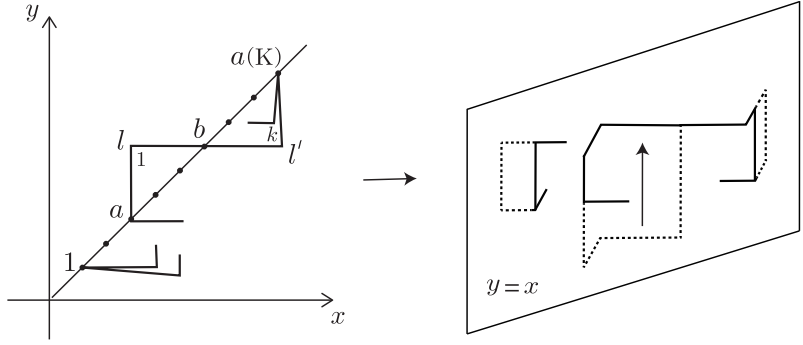}
\end{center}
\caption{Lattice presentation with $4$ sticks reduced} \label{fig4}
\end{figure}

Now we are ready to prove the main theorem.

\begin{proof}[Proof of Theorem \ref{thm:main}]
Let $K$ be a nontrivial knot which is not the trefoil knot.
Suppose that $K$ is not $(n+1,n)$-torus knot.
By Lemma \ref{lem:-4} and Theorem \ref{thm:ac},
$K$ can be realized as a properly leveled lattice knot
by using $c(K)$ $x$-sticks, $c(K)+1$ $y$-sticks, and $c(K)+1$ $z$-sticks.
By using Lemma \ref{lem:leveled},
this lattice knot $K$ consists of at most
$\frac{c(K)^2-1}{2} + \frac{(c(K)+1)^2}{2} + \frac{(c(K)+1)^2}{2} = \frac{3c(K)^2+4c(K)+1}{2}$ edges
if $c(K)$ is odd,
or $\frac{c(K)^2}{2} + \frac{(c(K)+1)^2-1}{2} + \frac{(c(K)+1)^2-1}{2} = \frac{3c(K)^2+4c(K)}{2}$ edges
if $c(K)$ is even.
These equations imply that the lattice knot $K$ consists of at most
$\frac{3}{2}c(K)^2 + 2c(K) + \frac{1}{2}$ edges.

Moreover if $K$ is a non-alternating prime knot,
then $c(K)-2$ $x$-sticks, $c(K)-1$ $y$-sticks, and $c(K)-1$ $z$-sticks are sufficient.
The same calculations show that the lattice knot $K$ consists of
$\frac{3}{2}c(K)^2 - 4c(K) + \frac{5}{2}$ edges.

Suppose that $K$ is $(n+1,n)$-torus knot for $n \geq 3$.
These are well-known facts that $c(K)=(n+1)(n-1)=n^2-1$ and $a(K)=(n+1)+n=2n+1$.
So $a(K) = 2 \sqrt{c(K)+1} +1$.
By Lemma \ref{lem:-2}, $K$ can be realized as a properly leveled lattice knot
by using $2 \sqrt{c(K)+1}$ $x$-sticks , $2 \sqrt{c(K)+1}$ $y$-sticks, and $2 \sqrt{c(K)+1} +1$ $z$-sticks.
By using Lemma \ref{lem:leveled} again,
this lattice knot $K$ consists of at most
$\frac{(2 \sqrt{c(K)+1})^2}{2} + \frac{(2 \sqrt{c(K)+1})^2}{2} + \frac{(2 \sqrt{c(K)+1}+1)^2 -1}{2}
= 6 c(K) + 2 \sqrt{c(K)+1} +6$ edges.
\end{proof}

\section{Conclusion}

Knots are commonly found in molecular chains such as DNA and proteins,
and they have been considered to be useful models for structural analysis of these molecules.
In the laboratory, the microscopic level molecules are more similar to rigid sticks
than a flexible rope.
In this context, we will consider a knot which consist of finite line segments
in the cubic lattice.
It is natural to ask what is the smallest number of atoms needed to construct
a nontrivial knotted molecule?
We propose upper bounds on the minimal length of any cubic lattice knots.
For small knots $3_1$, $4_1$ and $5_1$, their minimal lengths
are already known as 24, 30 and 34, respectively \cite{D1, SIADSV}.
Also the minimal length was numerically estimated for various knots \cite{HNRAV, JaP, SIADSV}.
In these papers the upper bounds of the minimal lengths of knots with crossing numbers
4, 5, 6, 7 and 8 are 30, 36, 40, 46 and 52,
and our upper bound formula gives the upper bounds as 32, 48, 66, 88 and 112, respectively.
These upper bounds are rather loose even for small knots,
and will be very loose for knots with larger number of crossings.
However, we need to point out that our upper bound formula is obtained by
purely analytic calculations instead of numerical calculations.

Finally we better mention the term ``minimum ropelength"
which is closely related to the minimal length.
The ropelength of a knot is defined to be the quotient of its length by its thickness,
where thickness is the radius of the largest embedded tubular neighborhood around the knot.
Indeed the minimum ropelength of a knot is less than or equal to two times of the
minimal length of the knot.
In \cite{CFM}, they found an upper bound of the minimum ropelength of a non-split link $L$
which is $1.64 c(L)^2 + 7.69 c(L) + 6.74$.
Also in \cite{DEY}, another upper bound $272 c(L)^{\frac{3}{2}} + 168 C(L)
+ 44 c(L)^{\frac{1}{2}} +22$ is proposed.

\end{document}